\documentclass[11 pt]{amsart}
\usepackage{graphicx}
\usepackage{geometry}
\usepackage{bm}
\usepackage{amssymb,amsmath,amsfonts,stmaryrd,amsthm,mathrsfs,bm}
\usepackage[dvipsnames]{xcolor}
\usepackage{color}
\usepackage[latin1]{inputenc}
\newcommand{\numberset}{\mathbb}
\newcommand{\N}{\numberset{N}}

\newcommand{\R}{\numberset{R}}

\newcommand{\dv}{\nabla\cdot}
\newcommand{\rot}{\nabla\times}
\newcommand{\ve}{\Vec{v}}

\newcommand{\dez}{\partial_z}

\newcommand{\grad}{\nabla}
\newcommand{\vx}{v^x}
\newcommand{\vz}{v^z}

\newcommand{\Ra}{R\!a\,}
\newcommand{\Prr}{P\!r \,}
\newcommand{\pr}{\frac{1}{\Prr}}

\newcommand{\be}{\beta}
\newcommand{\eb}{e^{\be z}}
\newcommand{\en}{e^{-\be z}}
\newcommand{\ebmm}{e^{-\frac{\be}{2}z}}
\newcommand{\ebmp}{e^{\frac{\be}{2}z}}

 \numberwithin{dummy}{section}
\newtheorem{theo}{Theorem}[section]
\newtheorem{lem}{Lemma}[section]

\newtheorem{defin}{Definition}[section]
\theoremstyle{definition}

\begin{document}

\title{A Lorenz Model for an Anelastic Oberbeck-Boussinesq System}


\maketitle
\qquad\qquad\qquad\qquad\qquad\author{Diego Grandi, Arianna Passerini, Manuela Trullo}
\footnote{Department of Mathematics and Computer Science, University of Ferrara, Via Machiavelli 30, 40121 Ferrara, Italy,
e-mail: {\tt\small grndgi@unife.it,ari@unife.it,manuela.trullo@edu.unife.it}}

\begin{abstract}
In an Oberbeck-Boussinesq model, rigorously derived, which includes compressibility, one could expect that the onset of convection for B\'{e}nard's problem occurs at a higher critical Rayleigh number. Since of the difficulties related to the new partial differential equations, with non constant coefficients and a non divergence-free velocity field, we show the increased stability of the rest state by exploring the related Lorenz approximation system.
\end{abstract}

\vspace{.7truecm}
\keywords{\emph{Keywords}: Oberbeck-Boussinesq, Lorenz model, Stability, Compressible fluids,  Anelastic approximation.}\\

%
%
%
\section{Introduction}

\vspace{.5cm}

Since long time, the Oberbeck-Boussinesq approximation \cite{B03,Cha61} is present in literature, as a relatively addressable system of partial differential equations which describes convection in a horizontal layer of fluid. The object of the study is the flow driven in the layer by a downward gradient of temperature $T$, due to a thermal exchange with the environment, resulting in a constant difference $\delta T$ between the boundaries. This is called B\'{e}nard's problem and studying it by the O-B system supplies some facilities. From the very beginning, as an approximation of the full set of Navier-Stokes-Fourier balance laws for natural convection in Newtonian fluids, the O-B system turned out to be a good model for engineers, but the theoretical explanation.

As a matter of fact, the velocity field $\ve$ in classic O-B system  satisfies $\dv\ve=0$ at all points and times. This feature allows to separate the system by the Helmoltz-Weyl decomposition and solving for $\ve$ without knowing about the pressure field $p$, which can be found at the end by using $\ve$ and $T$ as data. Notice that, classically, all the best, physically reasonable, regularity properties of Navier-Stokes solutions are found for isochoric flows. Furthermore, in the classic O-B system, an highly non linear term in the energy balance equation, proportional to the squared norm of the symmetric part of the velocity gradient, is neglected. Finally, it is somehow assumed that these flows are compatible with a solely temperature dependent density, as occurs for water which is incompressible in standard conditions. In particular, the thermal expansion gives a (linear) contribution to the dynamics only in the buoyancy term, so that all the coefficients of the system are constants.

The O-B system was derived as a formal limit of N-S-F only in the last decades, see for instance \cite{RRS96}. More recently, in \cite{GMR12, GR12}, it was shown that demanding for thermodynamical variables which do not depend on pressure, could be not compatible with thermodynamic principles. In particular, the Gibbs law could be not fulfilled  and stability in waves propagation would not come out, since of the loss of concavity of the chemical potential. Therefore, compressibility including generalizations of the O-B approximation are needed and they were studied from the well-posedness and stability point of view, for instance in \cite{CP, DP19, ACDLM}. On the other hand, all new approximation models should be derived by perturbative methods from the full set of balance equations. This can be found in \cite{RSV15, GP21} by keeping the condition of solenoidal $\ve$ even for gases, provided one chooses a suitable region of the three non dimensional parameter space, which are briefly introduced here below.

Let us assume that all flows meet the condition of small variations in density with respect to the lower plane, where the thermodynamical state is constant, so that density and temperature are $\rho=\rho_d$ and $T=T_d$ there. Thus, by taking as independent thermodynamical variables $T$ and $p$ (two are sufficient) and considering a linear expansion of the density, by \textit{compressibility} we mean the partial derivative of the density with respect to $p$ at constant $T$, evaluated in the reference state $(T_d, p_d)$ and divided by $\rho_d$. Such a quantity has the dimension of the inverse of $p$. Since a typical value of $p$ in convective phenomena could be the hydrostatic pressure $\pi=\rho_d g h$, where $g$ is the gravity acceleration and $h$ the height of the layer, we choose $\pi$ as reference value for the pressure (not necessarily deduced, as in in \cite{PR}, from the other reference units) and define a non dimensional compressibility factor
\begin{equation}\label{beta}
\beta:=gh \frac{\partial \rho}{\partial p}(T_d,p_d)\, .
\end{equation}
This should be added to the other two independent non dimensional parameters, which are sufficient to describe the classic incompressible convection, i.e. the Prandtl and Rayleigh numbers
\begin{equation}\label{eq:Pra}
 \Prr:=\frac{ \mu}{\rho_d\kappa}\qquad\qquad\qquad\Ra:=\frac{\rho_d\, \alpha_d \,|\delta T| h^3 g}{\mu\kappa}\, ,
\end{equation}
where the viscosity $\mu$ and the thermal diffusivity $\kappa$ hinder the convective motion, whereas $\alpha_d>0$, the thermal expansion coefficient, can trigger it. Precisely, the  definition is
$$\alpha_d=\frac{1}{\rho_d}\left|\frac{\partial \rho}{\partial T}(T_d,p_d)\right|\, .$$
The non dimensional counterpart of $\alpha_d$
\begin{equation}\label{alfa}
\alpha=\frac{|\delta T|}{\rho_d}\left|\frac{\partial \rho}{\partial T}(T_d,p_d)\right|\, ,
\end{equation}
is negligible with respect to $\beta$ and $\Ra$ in the \textit{anelastic O-B system} we are going to analyze in the present paper.

It can happen that in different regions of the parameter space the condition $\dv\ve=0$ could be relaxed in performing formal limits which lead to reliable approximations for compressible fluids: see \cite{Fei08} for the isothermal case, named \textit{anelastic approximation}, and see \cite{GP20} for the stratified O-B approximation we deal with. The problem with the equations in \cite{GP20} is that the balance of mass consists in the equation $\dv(\rho_e\ve)=0$, where $\rho_e$ is that density which corresponds to the \textit{rest state solution} and depends on $z$, $z$ the usual coordinate in the ascendant vertical direction, while it is constant when one derives the classic O-B system. In such a framework, stability results and well-posedness are harder to find, that is why in the present paper we face the problem of stability using the Lorenz model of the new O-B system, as done in \cite{DP20} for the model studied in \cite{CP}.

The Lorenz model is an ordinary non linear differential equation system, autonomous and homogeneous. It is mainly studied as the first and most famous example of chaotic system, but we are going to study only its first bifurcation in order to see for which critical value of $\Ra$ the rest state becomes unstable. The three unknowns of the system are the coefficients of the simplest Galerkin approximation solution which reproduces the expected behaviour of the full system: stability of the rest state for $\Ra$ sufficiently small, then, by increasing it, the onset of steady convection is observed as \textit{exchange of stability}, so that the rest state becomes a saddle point, at the end.

The objective of the paper is to prove the stabilizing effect of $p$ on the rest state solution of the model proposed in \cite{GP20}, at least in the related Lorenz model. As a matter of fact, in the role of state equation we use a first order Taylor expansion for density which, in non dimensional units ($\rho_d=1$, $|\delta T|=1$), is
\begin{equation}\label{steq}
\rho=1-\alpha (T-T_d)+\beta(p-p_d)\, ,
\end{equation}
with $\alpha$ and $\beta$ defined in (\ref{alfa}) and (\ref{beta}). As is evident by the sign, the compressibility correction in the buoyancy force has an opposite effect with respect to the thermal expansion (but for water in between $273K$ and $277K$), since the pressure makes the volume smaller. Thus, any contribution due to $\beta>0$ should increase the value of $\Ra$ beyond which convective motions are observable. It should be larger than the classic exact result $\Ra=\frac{27\pi^4}{4}$.

The structure of the paper is the following: in the next section, we present the new non dimensional PDE system describing convection and compare the rest state solution of the full N-S-F system, to which we append constitutive equation (\ref{steq}), with the rest state solution of the anelastic O-B; next, we choose the same boundary conditions allowing the classic exact result for the first stability transition; moreover, we look for a 2D solution and use a streamfunction $\psi$ to get the velocity field; finally, we introduce a new suitable complete basis which can be used to express both $\Delta\psi$ and the perturbation temperature $\tau$. The perturbation, which is the difference between $T$ and the linear stratified conduction profile $T_e=T_d-z$, enjoys Dirichlet conditions as $\Delta \psi$ does.

In the last section, the Lorenz system is finally derived. Since of the particular choice of the basis, whose main property is the diagonalization of the diffusive term (see \cite{P21}, where the existence is studied for the isothermal case), the Lorenz system can be reduced to the classic one by scaling the coefficients. It is exactly by means of the scaling equations that we prove that $r=1$, the critical value for Lorenz, corresponds to $\Ra(\beta)>\frac{27\pi^4}{4}$.


\section{Preliminary results}

\vspace{.5cm}

Let us specify all the typical physical quantities defining the non dimensional system in \cite{GP20}
\begin{equation}\nonumber \label{e:dimensionlessv}
t=\frac{h^2}{\kappa}\, ,\qquad
z=h\, ,\qquad
T=\delta T\, ,\qquad
\rho=\rho_d\, .
\end{equation}

Now, it is useful to stress how the constitutive equation (\ref{steq}) implies an \textit{environmental} stratified steady density $\rho_e$ which, as will be seen below, induces a specific inner product in the Hilbert space of the solutions.

By letting arbitrary the state equation which defines the local thermodynamical equilibrium, it is immediately verified by direct substitution that the N-S-F system in a layer
allows as a basic solution $(\ve_e,T_e,p_e,\rho_e)$
\begin{equation}\label{fullsol}
\ve_e=0\qquad\, T_e=T_d-z\qquad\, p_e=p_d-\int_0^zf(z)dz\qquad\, \rho_e=f(z)\, .
\end{equation}

Equation (\ref{steq}) gets rid of any freedom and one finds
\begin{equation}\label{eqsol}
\rho_e= \en + \frac{\alpha}{\be}(1-\en)\qquad\, p_e= p_d-\frac{1-\en}{\be}+\frac{\alpha}{\be}\left(\frac{1-\en}{\be}-z\right)\, .
\end{equation}
The system we study is the limit of the non dimensional form of N-S-F system as $\alpha\rightarrow 0$ when all the other non dimensional coefficients of the PDE system are kept constant. In particular, from $\Ra$ constant it follows clearly
$$\mathcal{G}:= \frac{\rho_d\, h^3 g}{\mu\kappa}=O\left(\frac{1}{\alpha}\right).$$

As a consequence, also the dissipative non linear term in the energy balance disappears from the equations as $\alpha$ tends to zero, because its coefficient is $\frac{\Prr}{\mathcal{G}}\Lambda$, where $\Lambda:= \frac{g\, h}{C_V|\delta T|}$ is the ratio between adiabatic and thermal gradients and $C_V$ denotes the specific heat at constant volume (see \cite{GP20} for all details).

In the limit (\ref{eqsol}) becomes
\begin{equation}\label{eqsol}
\rho_e= \en \qquad\, p_e= p_d-\frac{1-\en}{\be}\, ,
\end{equation}
while the adiabatic O-B system, from now on indicated by $(OB)_\be$, is
\begin{align*}
\textrm{(OB)}_\be\qquad
\left\{
\begin{array}{ll}
\displaystyle \dv \ve= \be \vz,\\
\displaystyle \en \, \frac{1}{\Prr}( {\ve}_{t}+\ve\cdot\nabla\ve)- \,\left(\be \gamma\grad \vz+ \Delta \ve\right)=-\nabla P+\left(\sqrt{\Ra}\,\tau-\beta P\right) \,\\
\en(\tau_t+\ve\cdot\grad\tau)-\Delta\tau=\sqrt{\Ra}\,\vz\, ,
\end{array}
\right.
\end{align*}
where, given the bulk viscosity $\zeta$, then $\gamma:=\frac{\zeta}{\mu} +\frac{1}{3}$. Here, just for symmetry between the two evolutionary equations, we scaled $\small \tau:=\sqrt{\Ra}(T-T_e)$ and set $P=p-p_e$. Partial derivatives are herein denoted by a subscript with the related variable.

Notice that the first equation in $(OB)_\be$ is verified if and only if
\begin{equation}\label{anel}
\dv (\en\, \ve)=-\be \en \vz+\en \dv \ve=0\, .
\end{equation}
Further, we can multiply the second and third equations in $(OB)_\be$ by $\eb$, set
\begin{equation}\label{Pigreco}
\Pi:=\eb \left(P- \,\gamma\beta v^z\right)\, ,
\end{equation}
and write in place of $(OB)_\be$ the equivalent system
\begin{align}\label{ansys}
\left\{
\begin{array}{ll}\medskip
\dv(\en \ve)=0 \\
\displaystyle  \frac{1}{\Prr} \,( \ve_t+\ve\cdot\grad\ve)- \eb  \Delta \ve =-\nabla \Pi + \eb\left(\sqrt{\Ra}\,\tau - \gamma\beta^2 v^z\vec{e}_3\right)\, \\
\tau_t+\ve\cdot\grad\tau-\eb \Delta\tau=\,\sqrt{\Ra}\,\vz
\end{array}
\right.
\end{align}

In order to compare the critical values of $\Ra$, we put ourselves in the same situation of classical O-B and Lorenz systems: we study (\ref{ansys}) in a 2D periodicity cell $\Omega=(0,\emph{l})\times(0,1)$ and append to it the following boundary conditions
\begin{equation}\label{bc}
\vz(x,0,t)=\vz(x,1,t)=0\qquad\qquad\vx_z(x,0,t)=\vx_z(x,1,t)=0\qquad\qquad\tau(x,0)=\tau(x,1)=0
\end{equation}
\begin{equation}\nonumber
\ve(0,z,t)=\ve(\emph{l},z,t)\qquad\qquad\tau(0,z,t)=\tau(\emph{l},z,t)
\end{equation}
which are impermeability conditions together with no-stress conditions (in the flat geometry of the layer) and periodicity along the horizontal direction $x$.

Moreover, we notice that with these conditions the system is invariant by dragging along the $x$-direction, so we assume that the mean value of $\vx$ is zero at all times:
\begin{equation*}
\int_{\Omega}\vx(x,z,t)\,dV=0\, ,
\end{equation*}
otherwise, the rest state solution would not be the only basic steady solution and any $\ve=V\vec{e}_1$ with constant $V$ would be steady solution as well, arbitrarily close to the rest state.

Actually, in two dimensions it is convenient the streamfunction formulation of problem (\ref{ansys}), which is obtained by applying the operator $\rot$ to both sides of the second equation and by replacing $\ve=\eb(-\psi_z\Vec{e}_1+\psi_x^j\Vec{e}_3)$ both in the second and in the third one. In a 2D domain such a form for $\ve$ automatically verifies the first equation in (\ref{ansys}), because of the streamfunction $\psi$.

Then, the second equation becomes a scalar equation: as a matter of fact, the system is invariant with respect to the direction $\Vec{e}_2$, normal to the strip where the system is considered and can be solved with no need of a third spatial variable $y$. The computation is elementary but long, see \cite{Tru} for all details. Therefore, we directly write the resulting final system.

Use will be made of a \textit{vorticity operator}, whose domain is specified below, which transforms $\psi$ in the function $-\omega\en$, where
\begin{equation}\label{vorticity}
    \omega := \vx_z - \vz_x= - \eb ( \Delta \psi + \be \psi_z)\, .
\end{equation}
Here, we use (\ref{vorticity}) just to represent shortly
system (\ref{ansys}) as equivalent to the two equations
\begin{align}\label{oursys}
    \left \{ \begin{array} {ll}
    & \pr (\omega_t + \be \eb \omega \psi_x + \eb (\psi_x \omega_z - \psi_z \omega_x) ) - \eb \Delta (\omega - \be \eb \psi_z  ) 
    = - \sqrt{\Ra} \eb \tau_x + \gamma \be ^2  e^{2\be z} \psi_{xx} \, ,\\
& \hspace{1cm} \\    
    &    \tau_t + \eb (\psi_x \tau_z - \psi_z \tau_x)  -\eb \Delta \tau = \sqrt{\Ra} \eb \psi_x\, .
\end{array} .
\right .
\end{align}

As shown in \cite{P21}, inserting condition (\ref{anel}) in the system implies to look for a new basis which diagonalizes the term with the highest derivative when constructing Galerkin approximation solutions, otherwise there is no hope to maintain the structure of the Lorenz model, as can be seen in the next section.

Here below, the suitable basis to express any $\ve$ in the solution of system $(\ref{ansys})$ is given as $\Vec{\psi}^j=\eb(-\psi_z^j\Vec{e}_1+\psi_x^j\Vec{e}_3)$, where $j$ is a multi-index for
$$\psi^j=\varphi^je^{-\frac{\be}{2}z}\, ,$$
with $\varphi^j$ eigenfunction of the usual Laplace operator (see \cite{G1}) with homogeneous boundary conditions in $z=0,1$ and periodic in $x$. Concerning this last feature,
the normalized orthogonal classic Fourier basis in $L^2(0,\emph{l})$ is now introduced, through a notation with a 2-valued index $i=\pm1$, i.e.
\begin{align}\label{per}
\phi^{1m}(x)=\sqrt{\frac{2}{\emph{l}}}\cos \left (\frac{2 \pi m }{\emph{l}}x \right )\qquad\,  \phi^{-1m}(x)=\sqrt{\frac{2}{\emph{l}}}\sin \left (\frac{2 \pi m }{\emph{l}}x \right )\, .
\end{align}
But the normalization constant $\sqrt{\frac{2}{\emph{l}}}$ should be replaced in case $m=0$ by $\sqrt{\frac{1}{\emph{l}}}$.
This notation, introduced first in \cite{CP}, can be very useful in calculations because the functions (\ref{per}) enjoy the property
\begin{align}\label{dex}
\frac{d\phi^{im}}{dx}=-i\frac{2 \pi m }{\emph{l}}\phi^{-im}\qquad\, i=\pm1\, .
\end{align}
Next, let us denote by
$$\chi^n(z):=\sqrt{2}\sin (n\pi z)\, ,$$
a well-known complete basis in $L^2(0,1)$. As a matter of fact, the one-dimensional Laplace operator $\mathcal{L}$ which makes the second derivative of all functions $\chi$ whose first and second derivative are in $L^2(0,1)$ and $\chi(0)=\chi(1)=0$, proves to be self adjoint: integrating by parts twice, one sees that $$\langle\chi_1,\chi_2''\rangle=\langle\chi_1'',\chi_2\rangle\,  ,$$ can be obtained if and only if $\chi_1$ and $\chi_2$ are in the same space. Further, $\mathcal{L}$ is invertible since no linear function vanishes in both $z=0,1$. Moreover, $\mathcal{L}^{-1}$ is compact since, integrating by parts and using Schwarz inequality:
$$\mathcal{L}(\chi)=f \qquad \Rightarrow \qquad \|\chi'\|_2\leq \|f\|_2\, ,$$
as can be seen by performing the scalar product of the equation with $\chi$, integrating by parts and using Schwarz inequality. Therefore, the image of $\mathcal{L}^{-1}$ is compactly embedded in $L^2(0,1)$ because of Rellich Theorem (which easily follows, in one dimension, from Ascoli-Arzel\`{a} theorem). Finally, by the Spectral Theorem the eigenfunctions $\chi^n(z)$ of $\mathcal{L}$, corresponding to eigenvalues $-n^2\pi^2$, are a complete basis.

Now, let us write
\begin{align}\label{bbasis0}
\varphi^j:= \varphi ^{imn} (x,z):=\phi^{im}(x)\chi^n(z)\qquad\, i=\pm1\, \,  m,n \in \N \, , \, \, n \neq 0\, ,
\end{align}
The $\varphi^j$'s in (\ref{bbasis0}) provide a basis for $\ve$ in the solution of the classic O-B problem ($\beta=0$), with the same boundary conditions (\ref{bc}), through the relation
$$\Vec{\psi}^j=-\varphi_z^j\Vec{e}_1+\varphi_x^j\Vec{e}_3\, ,$$ while the $\varphi^j$'s themselves are a basis for the temperature field too.

Now, we are going to show that the set of $\psi^j$ defined by
\begin{align}\label{bbasis}
\psi^j:= \psi^{imn} (x,z):=\phi^{im}(x)\chi^n(z)\ebmm\qquad\, i=\pm1\, \,  m,n \in \N \, , \, \, n \neq 0\, .
\end{align}
provides a suitable basis for $\ve$ in $(\ref{ansys})$ and is a basis for $\tau$ too.

It is then necessary to deduce boundary conditions (\ref{bc}) from basis (\ref{bbasis}): it can be seen by direct computation that $\Vec{\psi}^{j}=\eb (-\psi_z^j\Vec{e}_1+\psi_x^j\Vec{e}_3)$ enjoys boundary conditions
$$(\Vec{\psi}^{j})^{z}=0\,,\ \qquad   (\Vec{\psi}^{j})^{x}_z=0\qquad \ at\, \, \,  z=0,1\,\, .$$

Indeed, the required boundary conditions for $\Vec{\psi}^j$ follows from (\ref{dex})
and, further, from
\begin{align}\label{newbc}
  &\psi^{imn}(x,0)=\psi^{imn}(x,1)= 0\, , \\
  &\psi_{zz}^{imn}(x,0)+ \beta \psi_{z}^{imn}(x,0) =\psi_{zz}^{imn}(x,1)+ \beta \psi_{z}^{imn}(x,1) = 0\nonumber\, .
  \end{align}
These last are immediate consequence of definition (\ref{bbasis}).

Finally, let us define $\mathcal{C}^{\infty}(\Omega)$ as the set of the restrictions to $\Omega$ of $C_0^{\infty}$-functions whose support is anywhere in the strip between $z=0$ and $z=1$. Next, we denote with $W^{r,2}_0(\Omega)$ the closure of $\mathcal{C}^{\infty}(\Omega)$ in the usual norm $\| \cdot \|_{r,2,\Omega}$ of the Sobolev spaces, $r\in \N$ (for $r=0$ one gets $L^2(\Omega)$).

Still, by direct computation and through (\ref{dex}), see again \cite{Tru}, one can easily verify that $\psi^{imn}$ are eigenfunctions of the operator defined by (\ref{vorticity}) in the domain $W^{2,2}_0(\Omega)$: in fact the $L^2$-norm of the Laplacian and those of the second derivatives are equivalent in our $\Omega$, see for instance \cite{G1}.
Finally, we can proof that the couple $\psi$ and $\tau$ solving system $(\ref{oursys})$, under conditions (\ref{newbc}) for $\psi$, while for $\tau$ the first condition in (\ref{newbc}) is sufficient, can be both expressed in $L^2(\Omega)$ just by these eigenfunctions. The advantage is that, by writing $(\ref{oursys})$ as series, the linear coupling terms in the two equations are expressed by the same basis function.
\begin{theo} \label{BBasis}
The set of all the eigenfunctions $\psi^{imn}\in W^{2,2}_0(\Omega)$ of the vorticity operator:
\begin{equation*}
\Delta \psi^{imn}+\be \psi^{imn}_z=\mu^{mn}\psi^{imn}\qquad\qquad \mu_{mn} =-\left(\frac{\be^2}{4} + \frac{4 m^2 \pi ^2 }{\emph{l}^2}+ n^2 \pi ^2\right) .
\end{equation*}
is a complete basis for $L^2(\Omega)$. Moreover, the eigenfunctions are orthonormal with respect to the inner product
\begin{equation*}
\langle\psi^{imn},\psi^{jkl}\rangle_{_{\be}}:=\int_{\Omega}\psi^{imn}\eb\psi^{jkl}\, d\Omega=\delta_{ij}\delta_{mk}\delta_{nl}
\end{equation*}
\end{theo}
\begin{proof}
First, one can verify by direct elementary calculations that the formulas in the statement follows from (\ref{bbasis}). Then, the shortest way to show that these eigenfunctions are a complete basis is the following: consider a periodic in $x$ arbitrary $\psi\in L^2(\Omega)$ and represent it, a.e. in $z$, as a classic Fourier series with its basis functions $\phi^{im}$, which belongs to $C^2(0,\emph{l})$, and general coefficients depending on $z$. So, for instance, $\psi$ in $(\ref{oursys})$
can be written as
$$\psi(x,z,t)=\sum_{i,m}C^{im}(z,t)\phi^{im}(x)\qquad\, C^{im}(z,t)=\langle\phi^{im},\psi\rangle\, ,$$
where the inner product is the usual one in $L^2(0,\emph{l})$.
Therefore, from (\ref{bbasis}), it follows that the proof can be completed by showing that $C^{i,m}(z,t)$, which necessarily belong to $L^2(0,1)$, can be univocally written as series of the functions
$\chi^n(z){\ebmm}\, .$
To this end, we implicitly define $g^{im}$ by
$$C^{im}(z,t)=C^{im}(z,t)\ebmp\ebmm=g^{im}(z,t)\ebmm\, .$$
Since $\ebmp$ is bounded in $(0,1)$, then $g^{im}\in L^2(0,1)$ and one can write the identity
$$g^{im}=\sum_n\langle\chi^n,g^{im}\rangle\, \chi_n=\sum_n\langle\chi^n\ebmm,\eb g^{im}\ebmm\rangle\, \chi_n=\sum_n\langle\chi^n\ebmm, C^{im}\rangle_{_{\be}}\, \chi_n\, .$$
Let us multiply both sides by $\ebmm$:
$$C^{im}(z,t)=\sum_n\langle\chi^n(z)\ebmm, C^{im}(z,t)\rangle_{_{\be}}\, \chi_n(z)\ebmm\, $$
and, by deriving the (equivalent) norm from the new inner product, we have also
$$\|C^{im}\|_{_{\be}}^2=\sum_n\langle\chi^n\ebmm, C^{im}\rangle_{_{\be}}^2\, .$$
Finally, we can substitute the expression for $C^{im}(z,t)$ in that for $\psi(x,z,t)$ and, setting for brevity $A^{imn}(t)=\langle\chi^n\ebmm, C^{im}\rangle_{_{\be}}\, ,$ obtain
$$\psi(x,z,t)=\sum_{i,m,n}A^{imn}(t)\phi^{im}(x)\chi_n(z)\ebmm\, ,$$
which completes the proof since of definition (\ref{bbasis}).
\end{proof}

\section{The compressible Lorenz model}
\vspace{.5cm}

For \textit{approximation solution} of the non linear initial boundary value problem we mean a kind of finite dimensional tentative solution, which is a linear combination of $n$ elements of the basis through unknown time functions as coefficients. At the initial time the approximation is equal to the corresponding truncation of the series representing the initial data in the Hilbert space. Such truncation provides initial conditions for the coefficients of the approximation solution. Next, the Galerkin solution is substituted in the PDE equations, which are in turn projected into the span of the $n$ functions, so getting an ODE system which is a Cauchy problem for the coefficients. Nevertheless, since the non linear operator in the original PDE does not commute with the projector, \textit{Galerkin's solutions} are not solutions of the original system even though one chose a truncated initial data. A real (weak) solution of the original system could possibly be the limit of a Galerkin sequence of approximation solutions lying in spaces of increasing dimensions. We do not know whether such a solution exists.

Nevertheless, we are going to try a first study of the stability of the rest state in system (\ref{ansys}) by means of the Lorenz solutions, as usual constructed by means of only three basis functions: one to describe the velocity field (so loosing the non linear term in the first equation) and two for the temperature perturbation describing both the transport of energy and the decay (no matter how high $\Ra$ is) of any initial perturbation independent of $x$. 

For the convenience of the reader, we rewrite here the system 
\begin{align}\nonumber
    \left \{ \begin{array} {ll}
    & \pr (\omega_t + \be \eb \omega \psi_x + \eb (\psi_x \omega_z - \psi_z \omega_x) ) - \eb \Delta (\omega - \be \eb \psi_z  )
    = - \sqrt{\Ra} \eb \tau_x + \gamma \be ^2  e^{2\be z} \psi_{xx} \, ,\\
& \hspace{1cm} \\    
    &    \tau_t + \eb (\psi_x \tau_z - \psi_z \tau_x)  -\eb \Delta \tau =\sqrt{\Ra} \eb \psi_x\, .
\end{array} .
\right .
\end{align}

Let us choose
\begin{align}
    \left \{ \begin{array} {ll}
& \psi = A(t) \psi^{-111}\\
& \tau = B(t) \psi^{111} + C(t) \psi^{102}
\end{array} .
\right .
\end{align}
Hence, one can rewrite the equations in the system as $\psi$ and $\tau$ were actually the solutions, so that from definition (\ref{vorticity}):
\begin{equation*}
    \omega = -A(t) \mu^{11} \eb \psi^{-111}\, .
\end{equation*}
By replacing this result, it can be seen that in each term a \textit{weight} $\eb$ is present, but the first one in the second equation. As a consequence, when projecting the linear terms one can take advantage, at least a little bit, of the orthogonality relations in Theorem 2.1. 

The term with the highest derivatives in the first equation implies very long computations. In spite they are elementary, we fully write them at least once. Thus, by setting (for the sake of brevity) $\lambda=-\pi^2-\frac{\be^2}{4}$ and $\mu^{11}=-\mu$, we perform the usual scalar product in $L^2(\Omega)$:
\begin{align*}
   \langle\eb \Delta (\omega - \be \eb \dez \psi),  \psi^{-111} \rangle =\langle \Delta (\omega - \be \eb \dez \psi),  \psi^{-111} \rangle _\be =
\end{align*}   
\begin{align*}   
    =\!-A\frac{4}{\emph{l}}(\be ^2\lambda\! +\! \mu ^2 )\! \!\int_0^{\emph{l}}\! \!\sin ^2\!\!\left(\frac{2\pi x}{\emph{l}}\right)\!\!  dx \!\!\int_0^1\!\!  \sin ^2 \!(\pi z) \eb \! dz\! +
    \! A\frac{8}{\emph{l}}\pi \be \mu \!\!\int_0^{\emph{l}}\!\! \sin ^2\!\!\left (\frac{2\pi x}{\emph{l}} \right) \!dx\!\! \int_0^1\!\! \eb \sin (\pi z)\! \cos (\pi z)  dz\! = \\
    = -A (\be ^2\lambda + \mu ^2 ) \int_0^1 \left (1- \cos (2\pi z ) \right ) \eb \, dz + A 2\pi \be \mu \int_0^1 \eb \sin (2\pi z) \, dz = \\
    = -A (\be^2 \lambda + \mu ^2 ) \left ( \frac{e^ \be -1}{\be} - \int _0^1 \cos (2\pi z) \eb \, dz \right ) +A \pi \be \mu  \, 4\pi \frac{1-e^\be}{\be ^2 + 4 \pi ^2} = \\
    =-A (\be ^2 \lambda + \mu ^2 ) \left ( \frac{ e ^ \be -1}{\be} - \frac{\be (e^\be - 1)}{\be ^2 + 4\pi ^2} \right ) +A \be \mu 4\pi ^2 \frac{1-e^\be}{\be ^2 + 4 \pi ^2} = \\
    = -A (\be ^2 \lambda + \mu ^2 ) \left ( e^\be -1\right) \frac{\be ^2 + 4\pi ^2 - \be ^2}{\be (\be ^2 + 4\pi ^2)} + A \be \mu 4\pi ^2 \frac{1-e^\be}{\be ^2 + 4 \pi ^2} = \\
    = -A (\be ^2 \lambda + \mu ^2 ) 4 \pi ^2 \frac{e^\be -1}{\be (\be ^2 + 4 \pi ^2)} + A \be \mu 4 \pi ^2 \frac{1-e^\be}{\be ^2 + 4 \pi ^2} = \\
    = A (\be ^2( \lambda + \mu) +\mu ^2   ) 4 \pi ^2 \frac{1-e^\be }{\be (\be ^2 + 4 \pi ^2)} = \\
    =A \left(\be ^2 \frac{4 \pi ^2}{\emph{l}^2} +\mu ^2 \right) 4 \pi ^2 \frac{1-e^\be }{\be (\be ^2 + 4 \pi ^2)}\, .
\end{align*}
The analogous term in the second equation must be projected on two different basis function:
\begin{equation*}
   \langle \eb \Delta \tau , \psi^{111} \rangle  =  \langle \Delta \tau , \psi^{111} \rangle _ \be = B(t)\left ( \frac{\be^2}{4}- \pi ^2  -\frac{4 \pi ^2}{\emph{l}^2}\right ) \, ,
\end{equation*}
\begin{equation*}
  \langle\eb \Delta \tau , \psi^{102} \rangle =  \langle \Delta \tau , \psi^{102} \rangle _ \be = C(t)\, 2\left (\frac{\be ^2}{8} - 2\pi ^2 \right ) \, .
\end{equation*}
The term with partial time derivative is easily projected from the first equation
\begin{equation*}
    \langle \omega_t, \psi^{-111}\rangle=  \Dot{A}(t) \mu\, ,
\end{equation*}
while some extra computations are needed for the second equation, since the weight $\eb$ is missing:
\begin{align*}
    \langle  \tau_t , \psi^{111} \rangle  &=  \Dot{B}(t) \frac{4}{\emph{l}} \int _0^{\emph{l}} \cos^2 \left (\frac{2\pi x}{\emph{l}}  \right) \, dx \int_0^1 \sin ^2 (\pi z) \en \, dz = \\
    &=\Dot{B}(t) \int _0^1 \left ( 1 - \cos (2\pi z) \right) \en \, dz = \Dot{B}  \left ( \frac{1- e ^{-\be}}{\be} - \frac{\be (1- e ^{-\be})}{\be^2 + 4\pi ^2 } \right )= \\
    &= \Dot{B}(t) (1- e ^{-\be}) \frac{\be^2 + 4\pi ^2 - \be^2}{\be (\be^2 + 4\pi ^2)} = \Dot{B} \frac{1- e^{-\be}}{\be}\frac{4 \pi^2}{\be ^2 + 4 \pi ^2}\, ,
\end{align*}
and analogously one gets
\begin{equation*}
    \langle \tau_t, \psi^{102}\rangle=  \Dot{C}(t) \frac{1- e^{-\be}}{\be}\frac{16 \pi^2}{\be ^2 + 16 \pi ^2}\, .
\end{equation*}
However, from these results one can already see that the ordinary orthogonality relations keep in normal form the ODE system we are looking for.

Now, let us project the linear terms at the right hand side. In the first equation, one has
\begin{equation*}
     - \sqrt{\Ra} \langle\eb \tau_x ,  \psi^{-111}\rangle  = - \sqrt{\Ra} \langle \tau_x ,  \psi^{-111}\rangle _\be=  \sqrt{\Ra} B(t) \frac{2\pi}{\emph{l}} \|\psi^{-111}\|^2_\be = \sqrt{\Ra} B(t) \frac{2\pi} {\emph{l}} \, ,
\end{equation*}
and also
\begin{align*}
    \langle \gamma \be ^2  e^{2\be z} \psi_{xx} , \psi^{-111} \rangle &= \langle \gamma \be ^2  e^{\be z} \psi_{xx} , \psi^{-111}\rangle_\be=- A(t) \gamma \be ^2 \frac{4 \pi^2}{\emph{l}^2} \int_0^{\emph{l}} \sin ^2\left (\frac{2\pi x}{\emph{l}} \right) \, dx \int_0^1 \sin ^2 (\pi z) \eb \, dz = \\
    &= -  A(t) \gamma \be  \frac{16\pi ^4}{\emph{l}^2}\frac{e^\be -1 }{\be ^2 + 4\pi ^2}\, .
\end{align*}
The last linear term is at the right hand side of the second equation, therefore we have to project it twice. However, since of (\ref{dex})
$$\sqrt{\Ra} \eb \psi_x=\sqrt{\Ra} A(t)\eb \psi^{-111}_x=\sqrt{\Ra} A(t)\eb \frac{2\pi}{\emph{l}}\psi^{111}\, ,$$
it is easy to see
\begin{equation*}
\sqrt{\Ra}    \langle  \eb \psi_x ,  \psi^{111}\rangle =\sqrt{\Ra}\langle   \psi_x , \psi^{111} \rangle _ \be    = \sqrt{\Ra} A(t) \frac{2\pi }{\emph{l}} \|\psi^{111}\|^2_\be = \sqrt{\Ra} A(t) \frac{2\pi }{\emph{l}}   \, ,
\end{equation*}
while, clearly,
\begin{equation*}
\sqrt{\Ra}    \langle  \eb \psi_x ,  \psi^{102}\rangle =0   \, .
\end{equation*}
Concerning the non linear terms, the one in the first equation projects into zero, as expected by using only one basis function to approximate $\ve$. Then, we write the non linear term in the second equation: for the elementary but long calculations leading to the two projections of such term, we again refer the reader to \cite{Tru} and write the results. Thus, one considers
\begin{align*}
    &\eb (\tau_z \psi_x -\tau_x \psi_z) = \\
    &=\eb A \frac{2\pi }{\emph{l}}\sin (\pi z) \en \Bigg(2 B \left (\pi \cos (\pi z)- \frac{\be}{2} \sin (\pi z) \right) +C \sqrt{2}\cos  \left (\frac{2\pi x}{L} \right) \left (2\pi \cos (2 \pi z) - \frac{\be}{2} \sin (2\pi z)\right ) \Bigg)\, ,
\end{align*}
and its projections:
\begin{align*}
    & \langle \eb(\tau_z \psi_x -\tau_x \psi_z) , \psi^{111} \rangle= \\
    &= AC\sqrt{\frac{2}{\emph{l}}} \frac{4\pi ^2}{\emph{l}} (e^{-\frac{\be}{2}} - 1 ) \frac{\be ^2 + 64 \pi ^2- \be ^2 }{\be (\be^2 + 64 \pi ^2 ) } = AC  \frac{e^{-\frac{\be}{2}}-1}{\be} \frac{ 64 \pi ^4}{\be^2 + 64 \pi ^2  } .
\end{align*}
\begin{align*}
   &\langle \eb(\tau_z \psi_x -\tau_x \psi_z) , \psi^{102 } \rangle= \\
   &=AB\sqrt{\frac{2}{\emph{l}}} \frac{4\pi ^2}{\emph{l}} \frac{e^{-\frac{\be}{2}}-1}{\be}  \left (\frac{4\be ^2}{\be^2 + 16 \pi^2} - \frac{3\be^2}{\be^2 + 64\pi ^2}   -1\right )\, .
\end{align*}
Having collected all terms, we can substitute them in the system and write the three ODE of the Lorenz system:
\begin{align*}
    \label{sistemastratificatoABC}
    &\Dot{A}= -\frac{4 \pi ^2 \Prr}{\mu (\be ^2 +4\pi ^2 )} \left (\mu ^2 + \be ^2 \frac{4 \pi ^2}{\emph{l}^2}  + \gamma\be ^2 \frac{4 \pi ^2}{\emph{l}}  \right ) \frac{e^\be - 1}{\be} A + \frac{2 \pi \Prr \sqrt{\Ra}}{\mu \emph{l}} B, \\
    &\Dot{B}= \sqrt{\frac{2}{\emph{l}}}\frac{\be ^2 + 4 \pi ^2}{\be ^2 + 16 \pi ^2} \frac{64 \pi ^2}{1 +e^{-\frac{\be}{2}}} \frac{1}{\emph{l}} AC - \frac{\be ^2 + 4 \pi ^2}{4 \pi ^2} \frac{\be}{1 - e^ {- \be}} \left (\pi ^2 + \frac{4 \pi ^2}{\emph{l}^2} - \frac{\be ^2}{4} \right ) B + \sqrt{\Ra}\frac{\be ^2 + 4 \pi ^2}{2 \pi \emph{l}} \frac{\be}{1 - e^ {- \be}} A, \\
    &\Dot{C} =\sqrt{\frac{2}{\emph{l}}} \frac{\be ^2 + 16 \pi ^2}{8 \emph{l}} \frac{2}{1 + e^{-\frac{\be}{2}}} \left (\frac{4\be ^2}{\be ^2 + 16 \pi ^2} - \frac{3\be ^2}{\be ^2 + 64 \pi ^2} -1 \right ) AB + \frac{\be ^2 + 16 \pi ^2}{8 \pi ^2} \frac{\be}{1 - e^ {- \be}} \left ( \frac{\be ^2}{8} - 2 \pi ^2 \right )C,
\end{align*}
In order to simplify the calculus, we rename the coefficients as follows
\begin{align*}
&  \Dot{A} = e_1 A + e_2 B \\
&  \Dot{B} =  e_3 AC + e_4 B + e_5 A \\
& \Dot{C} = e_6 AB + e_7 C
\, .
\end{align*}
As a matter of fact, our aim is to perform a scaling of this homogeneous system to transform it in the well-known classic Lorenz model, that is
\begin{align*}
&  \Dot{X} = \sigma (Y-X) \\
&  \Dot{Y} =  - XZ + rX -Y \\
& \Dot{Z} = XY - \delta Z
\end{align*}
and use all its known stability features.

Besides the free parameters implied by the unknown, i.e.
\begin{equation*}
    X=aA \qquad Y=bB \qquad Z=cC \qquad, 
\end{equation*}
we have a further parameter available by scaling the time as $s:=t\, d$, so that
\begin{equation*}
\frac{d}{dt}=\frac{ds}{dt}\frac{d}{ds}\, .
\end{equation*}
Therefore, the system can be rewritten in the unknown $(X(s),Y(s),Z(s))$
\begin{align*}
    & \frac{dX}{ds} = \frac{e_1}{d}X + \frac{e_2 \,a}{bd} Y \\
    &\frac{dY}{ds}  = \frac{e_3 \, b}{d \, a \, c} XZ + \frac{e_4}{d} Y + \frac{e_5 \, b}{a \, d} X  \\
    & \frac{dZ}{ds} = \frac{e_6 \, c}{d \, a \, b} XY + \frac{e_7}{d} Z.
\end{align*}
Then, if we set
\begin{equation*}
    \sigma = \frac{e_2 \,a}{bd}= - \frac{e_1}{d}\, ,
\end{equation*}
we obtain the first equation in the wanted form by only fixing
\begin{equation*}
    \frac{a}{b} = - \frac{e_1}{e_2}.
\end{equation*}
Next, necessarily, if one wants that the second coefficient in the second equation is $-1$, then

\begin{equation*}
    d = - e_4
\end{equation*}
and the second equation becomes
\begin{equation*}
    \frac{dY}{ds} = \frac{e_3 \, e_2}{e_4 \, e_1 \, c} XZ - Y +\frac{e_5 \, e_2}{ e_4  \, e_1} X\, .
\end{equation*}
Let us point out that our need is actually to identify the coefficient $r$ which, as explained below, is crucial for stability, and now we have it:
\begin{equation}\label{Raa}
r =  \frac{e_5 \, e_2}{ e_4  \, e_1} 
\end{equation}

Nevertheless, for the sake of completeness, we finish writing the scaling equations: the first coefficient in the second equation must be $-1$
\begin{align*}
     \frac{e_3 \, e_2}{e_4 \, e_1 \, c} = -1 \quad \iff \quad c= - \frac{e_3 \, e_2}{e_4 \, e_1 } \, .
\end{align*}
The third equation becomes
\begin{equation*}
    \frac{dZ}{ds} = -\frac{e_6 \, e_3 \, e_2 ^2}{e_4 ^2 \, e_1^2 \, b^2 } XY - \frac{e_7}{e_4} Z\, ,
\end{equation*}
so that, by comparing it with Lorenz third equation, we find
\begin{align*}
     -\frac{e_6 \, e_3 \, e_2 ^2}{e_4 ^2 \, e_1^2 \, b^2 } = 1 \quad & \iff \quad  b^2 = -\frac{e_3 \, e_6 \, e_2 ^2}{e_1 ^2 \, e_4^2  } \\
    & \frac{e_7}{e_4} = \delta
\end{align*}
and it can be easily verified that
 $b^2 >0$ holds true.
  
Finally, it is interesting to stress how the coefficient $r$ depends only on the coefficients of linear terms. As a matter of fact, in dynamical systems the linear part of the operator which defines the time derivative helps a lot. To this purpose, let us recall  
\begin{defin}
Given the Cauchy problem
\begin{align*} 
    \left \{ \begin{array}{ll}
         & \Dot{\Vec{y}} = \Vec{F} (\Vec{y}) \\
         & \Vec{y}(0) =\Vec{y}_0
    \end{array}, \right .  
\end{align*}
if $\vec{F} (\vec{y_0}) = \vec{0}$, with $\vec{F}\in \R^n$ continuous in $\R^n$, then the problem allows the unique solution $\vec{y}(t) = \vec{y_0}$ and $\vec{y_0}$ is called a critical point
\end{defin}
If $r\leq 1$ the origin is the only critical point for Lorenz system, otherwise one has also these two
$$P_{\pm} = ( \pm \sqrt{\delta (r-1)}, \pm \sqrt{\delta (r-1)}, r-1)\, .$$
If $\Vec{y}_0$ is different from the origin, by a translation one can always study it as it were the origin: this can be done by means of the following standard result
 
\begin{theo} 
Assume that the origin is a critical point for the system above and consider the linear expansion
$$\Vec{F} (\Vec{y}) = J \Vec{y} + \Vec{\sigma} (\Vec{y}) \|\Vec{y} \|\qquad \displaystyle \underset{\|\vec{y}\| \to 0}{\lim} \Vec{\sigma} (\Vec{y}) =0\, ,$$
where the Jacobian $J$ is evaluated in the origin. If all the eigenvalues of $J$ have negative real part then $\Vec{y}(t)=\Vec{0}$ is asymptotically stable.
If $J$ has at least an eigenvalue with positive real part, then $\Vec{y}(t)=\Vec{0}$ is unstable.
\end{theo}
For the Lorenz system, such Jacobian is
\begin{equation*}
    J = \left (\begin{matrix}
    - \sigma & \sigma & 0\\
    r & -1 & 0\\
    0 & 0 & - \delta
    \end{matrix} \right).
\end{equation*}
Its eigenvalue equation, besides $\lambda_3=-\delta$, allows also for the two (real) solutions
\begin{equation*} 
    \lambda_{1,2} = \frac{-(\sigma +1) \mp \sqrt{(\sigma + 1)^2 + 4\sigma (r-1)}}{2}\, .
\end{equation*}
Here, the \textit{plus} one starts to be positive as soon as $r>1$. One can also prove that the new critical points are now stable.

For this reason, let us write the expression of $r$: from $(\ref{Raa})$

 \begin{align*}
    &r=\frac{2 \pi \Prr \Ra}{\mu \emph{l}} \frac{\be ^2 + 4 \pi ^2}{2 \pi \emph{l}}  \frac{\be}{1 - e^ {- \be}} \frac{ \mu (\be ^2 + 4 \pi ^2 )}{4 \pi ^2 \Prr} \frac{\be}{e ^ \be - 1} \frac{1}{\mu ^2 +\be ^2 \frac{4 \pi ^2}{\emph{l}^2} + \gamma \be ^2 \frac{4 \pi ^2}{\emph{l}} } \frac{4 \pi ^2}{\be ^2 + 4 \pi ^2} \frac{1 - e^ {- \be}}{\be}
    \frac{1}{\pi ^2 +\frac{4 \pi ^2}{\emph{l}^2} - \frac{\be ^2}{4}} = \\
    & = \frac{\Ra}{\emph{l} ^2} (\be ^2 + 4 \pi ^2 ) \frac{\be}{e ^ \be - 1} \frac{1}{\mu ^2 +\be ^2 \frac{4 \pi ^2}{\emph{l}^2} + \gamma \be ^2 \frac{4 \pi ^2}{\emph{l}} } \frac{1}{\pi ^2 +\frac{4 \pi ^2}{\emph{l}^2} - \frac{\be ^2}{4}} .
\end{align*}
 Next, we set $r=1$ and deduce the critical value $\Ra^*$ for the anelastic model 
 \begin{align*}
    \Ra^*_{\be} &= \frac{\emph{l}^2}{\be ^2 + 4 \pi ^2} \frac{e ^ \be - 1}{\be} \left (\mu ^2 +\be ^2 \frac{4 \pi ^2}{\emph{l}^2} + \gamma \be ^2 \frac{4 \pi ^2}{\emph{l}}  \right ) \left (\pi ^2 + \frac{4 \pi^2}{\emph{l}^2} - \frac{\be ^2}{4} \right ) = \\
    & = \frac{\emph{l}^2}{\be ^2 + 4 \pi ^2} \frac{e ^ \be - 1}{\be} \left ( \left ( \pi ^2 + \frac{4 \pi^2}{\emph{l}^2} + \frac{\be ^2}{4}\right ) ^2 +\be ^2 \frac{4 \pi ^2}{\emph{l}^2} + \gamma \be ^2 \frac{4 \pi ^2}{\emph{l}}  \right ) \left (\pi ^2 + \frac{4 \pi^2}{\emph{l}^2} - \frac{\be ^2}{4} \right ).
\end{align*}
\section{Conclusions}
By developing in Taylor series with respect to $\be$ and considering only first order terms, we can conclude what follows
\begin{itemize}
    \item if $\be =0$, as is known
    \begin{equation*}
        \Ra^*=\frac{\alpha ^3\emph{l}^2}{4\pi^2}
    \end{equation*}
    \item if $\be >0$, but first order small in $\be$
    \begin{equation*}
    \Ra^*_\be  \simeq \frac{\alpha ^3\emph{l}^2}{4\pi^2} \left (1 + \frac{\be}{2}\right )> \Ra^*\, ,
\end{equation*}
\end{itemize}
because
    \begin{equation*}
        \frac{e ^ \be - 1}{\be}  \simeq 1 + \frac{\be}{2}    \, .
    \end{equation*}
We have so proven that the stability of the rest state is increased by the factor $\be$.

\vspace{1cm}
\textbf{Conflict of interest} The authors declare that there is no conflict of interest regarding the publication of this article. 

\textbf{Data Availability} This manuscript has no associated data.

\end{document}